\documentclass[oneside]{amsart}
\usepackage[T1]{fontenc}
\usepackage[latin9]{inputenc}
\usepackage{amsthm}
\usepackage{amssymb}
\usepackage{graphicx}

\makeatletter
\numberwithin{equation}{section}
\numberwithin{figure}{section}
  \theoremstyle{plain}
  \newtheorem*{thm*}{\protect\theoremname}
  \theoremstyle{remark}
  \newtheorem*{rem*}{\protect\remarkname}
\theoremstyle{plain}
\newtheorem{thm}{\protect\theoremname}[section]
  \theoremstyle{definition}
  \newtheorem{defn}[thm]{\protect\definitionname}
  \theoremstyle{remark}
  \newtheorem{rem}[thm]{\protect\remarkname}
  \theoremstyle{plain}
  \newtheorem{prop}[thm]{\protect\propositionname}
  \theoremstyle{plain}
  \newtheorem{lem}[thm]{\protect\lemmaname}

\makeatother

  \providecommand{\definitionname}{Definition}
  \providecommand{\lemmaname}{Lemma}
  \providecommand{\propositionname}{Proposition}
  \providecommand{\remarkname}{Remark}
  \providecommand{\theoremname}{Theorem}
\providecommand{\theoremname}{Theorem}

\begin{document}
\global\long\def\ww#1{\mathbb{#1}}
\global\long\def\pp#1{\mbox{d}#1}
\global\long\def\cp{\ww CP^{2}}

\title{An example of a topologically non-rigid foliation of the complex
projective plane}

\author{Loïc TEYSSIER }

\date{November 2010}
\begin{abstract}
We give here an explicit example of an algebraic family of foliations
of $\cp$ which is topologically trivial but not analytically trivial.
This example underlines the necessity of some assumptions in Y. Ilyashenko's
rigidity theorem.
\end{abstract}

\address{Institut de Recherche Mathématique Avancée (UMR 7501), Université
de Strasbourg et C.N.R.S., 7 rue René Descartes, 67000 Strasbourg,
France}

\dedicatory{$\copyright$\textbf{2010, the author. Published by Oxford University
Press. All rights reserved.}}

\maketitle
\begin{center}
Institut de Recherche Mathématique Avancée\\
Université de Strasbourg (France)\\

\par\end{center}

\section{Introduction and presentation of the result}

The aim of this article is to provide examples of algebraic foliations
of $\cp$ which are not topologically rigid. A foliation on $\cp$
is defined in an affine chart $\ww C^{2}$ by a differential equation
\begin{eqnarray}
P\left(x,y\right)y' & = & Q\left(x,y\right)\label{eq:diff_eq}
\end{eqnarray}
where $P$ and $Q$ are complex polynomials. The space $\mathcal{C}_{d}$
of all such foliations with $P$ and $Q$ of a fixed degree $d$ is
a complex projective space of finite dimension, endowed with the natural
topology. A theorem of Yu. S. Ilyashenko \cite{I} states that except
maybe for a residual set, all foliations which are topologically conjugate
are in fact analytically (thus homographically) conjugate, \emph{i.e.
}the generic foliation is topologically rigid. This result was later
enhanced by various authors. In \cite{S} A. Scherbakov showed that
the set of topologically rigid foliations contains at least the complement
of a real analytic set of $\mathcal{C}_{d}$. An improvement was given
by X. G\'{o}mez-Mont and L. Ort\'{i}z-Bobadilla \cite{GO}, then
by L. Neto, P. Sad and P. Sc\'{a}rdua \cite{NSS}, showing that the
set of all topologically rigid foliation is at least a Zariski-dense
open set of $\mathcal{C}_{d}$. The argument boils down to proving
that non-solvable holonomy representation is the typical behaviour,
then applying Nakai's theorem \cite{N} or using other results on
density of orbits of (pseudo-)groups of local diffeomorphisms. \\

Examples of\emph{ }foliations of $\cp$ which are not topologically
rigid do not abound. First examples of non-rigid foliations of $\cp$
can be deduced from the work of N. Ladis \cite{L}, where the topological
classification of generic homogeneous equations \eqref{eq:diff_eq}
is achieved.

We wish to present rather simple foliations which are not rigid. They
belong to the class of Liouville-integrable foliations, whose holonomy
representation is solvable. The main purpose of this paper is to prove
their non-rigidity ``by hand'', by building explicitely homeomorphisms
between each members of the following family :
\begin{thm*}
Let $\Omega\subset\ww C$ be the domain defined by $\Omega:=\left\{ \alpha\in\ww C\,:\,\left|\alpha\right|<\frac{1}{10}\sqrt{\frac{\pi}{2}}\right\} $.
Each member of the family of linear differential equations in $\ww C^{2}$
: 
\begin{eqnarray}
x^{3}y' & = & y+x^{2}+\alpha x^{3}\,\,\,\,\,,\,\alpha\in\Omega\label{eq:eq_diff}
\end{eqnarray}
induces a foliation of $\cp$ which is not topologically rigid. More
precisely, they all are topologically conjugate to each other whereas
two equations with different $\alpha$'s are not locally analytically
conjugate near $\left(0,0\right)$.
\end{thm*}
In fact, a part of the second statement is given by P. M. Elizarov's
result \cite{E}, where he describes the local topological classification
of saddle-node equations, and by Martinet-Ramis' one \cite{MR} about
their local analytical classification. The cornerstone of the proof
here is to build an homeomorphism of $\ww C^{2}$ which extends to
the whole of $\cp$, which was not possible using Elizarov's purely
local construction. For technical reasons and the sake of briefness
it was necessary to choose $\Omega$ as given above, though the result
should be valid for all $\alpha\in\ww C\backslash\left\{ \pm1\right\} $.
As a real map, the homeomorphisms introduced here are piece-wise affine
but could be chosen $C^{\infty}$ outside $\left\{ x=0\right\} $
by taking small perturbations. Yet according to a rigidity result
of S. M. Voronin \cite{V} the homeomorphism cannot be made $C^{1}$
in any neighbourhood of this line since otherwise the differential
equations would be locally analytically conjugate.
\begin{rem*}
This family of differential equations is not an unfolding in the sense
of J.-F. Mattéi \cite{M} since if there existed some germ of a holomorphic
function $R$ such that $x^{3}\pp y-\left(y+x^{2}+\alpha x^{3}\right)\pp x+R\left(x,y,\alpha\right)\pp{\alpha}$
be integrable (as a $1$-form) then one would obtain
\begin{eqnarray*}
x^{3}\frac{\partial R}{\partial x}+\left(y+x^{2}+\alpha x^{3}\right)\frac{\partial R}{\partial y} & = & \left(1+3x^{2}\right)R-x^{6}\,.
\end{eqnarray*}
On the one hand this equation admits a unique formal solution \cite{T}.
On the other hand so is the case for the differential equation $x^{3}f'=\left(1+3x^{2}\right)f-x^{6}$.
Taking $R\left(x,y,\alpha\right):=f\left(x\right)$ thus yields the
only possible solution. Unfortunately the latter power series is divergent.
\end{rem*}

\section{\label{sec:Local}Local study of the saddle-node singularity}
\begin{defn}
When we say that two foliations are \emph{locally topologically conjugate}
(or simply topologically conjugate) near $\left(0,0\right)$ we mean
that there exists an open neighbourhood $\Delta$ of $\left(0,0\right)$
and an orientation-preserving homeomorphism $\varphi\,:\,\Delta\to\varphi\left(\Delta\right)$
fixing $\left(0,0\right)$ which sends a (trace on $\Delta$ of a)
leaf of one foliation into a leaf of the other. If moreover $\varphi$
is an analytic map we say that the vector fields are \emph{locally
analytically conjugate}. 
\end{defn}
In the following we will study foliations $\mathcal{F}$ of $\cp$
on subdomains $\Delta$ of $\cp$ and we will implicitly mean that
we consider the restriction of $\mathcal{F}$ to $\Delta$, \emph{i.e.}
the foliation whose leaves are the connected components of the trace
on $\Delta$ of the leaves of $\mathcal{F}$.

\subsection{What is known}

We first apply the linear change of variables $\left(y,\alpha\right)\mapsto\left(-i\pi y,\sqrt{\frac{2}{\pi}}\alpha\right)$
in order to transform the family of equations \eqref{eq:eq_diff}
into 
\begin{eqnarray}
x^{3}y' & = & y-\frac{1}{i\pi}x^{2}-\frac{\alpha}{i\sqrt{2\pi}}x^{3}\,.\label{eq:eq_diff_alpha}
\end{eqnarray}
This change of variables is performed to simplify the upcoming computations.
Let us denote by $\omega_{\alpha}$ the differential $1$-form representing
\eqref{eq:eq_diff_alpha} which, in the affine chart $\ww C^{2}=\left\{ \left(x,y\right)\right\} $,
can be written as 
\begin{eqnarray}
\omega_{\alpha}\left(x,y\right) & := & \left(y-\frac{1}{i\pi}x^{2}-\frac{\alpha}{i\sqrt{2\pi}}x^{3}\right)\pp x-x^{3}\pp y\,.\label{eq:diff_form}
\end{eqnarray}
 Such a differential form is integrable and induces a foliation on
$\cp$, which we denote by $\mathcal{F}_{\alpha}$, of saddle-node
type at $\left(0,0\right)$, having exactly one separatrix passing
through this point (namely $\left\{ x=0\right\} $). Notice that all
leaves of the foliation are transverse to the fibers of the natural
projection
\begin{eqnarray*}
\Pi\,:\, & \left(x,y\right) & \mapsto x
\end{eqnarray*}
except for the separatrices $\left\{ x=0\right\} \cup\left\{ x=\infty\right\} $.

Let us recall two classical results which our argument is partly based
upon.
\begin{thm}
\emph{(Elizarov, \cite{E}) }Let $E$ be the space of all saddle-node
foliations given by differential forms $\left(y+R\left(x,y\right)\right)\pp x-x^{3}\pp y$,
where $R$ is a germ of a holomorphic function at $\left(0,0\right)$
with $R\left(0,0\right)=0$ and $\frac{\partial R}{\partial y}\left(0,0\right)=0$.
This space splits into $E_{1}$ and $E_{2}$ according to whether
a given foliation has one or two separatrices through $\left(0,0\right)$.
\begin{enumerate}
\item The quotient $E_{1}/_{top}$ of local topological equivalence classes
has cardinality $2$. 
\item Two foliations in $E_{2}$ are locally topologically conjugate if,
and only if, so are their ``weak'' holonomies, \emph{i.e. }the holonomies
computed on a transversal $\Pi^{-1}\left(x_{0}\right)$ by lifting
through $\Pi$ a generator of the fundamental group of the second
separatrix.
\end{enumerate}
\end{thm}
In fact the differential form \eqref{eq:diff_form} has such a simple
form that it is integrable by quadrature, so its invariant of topological
classification can be computed explicitly in terms of $\alpha$ (see
the end of \emph{\cite{MR}} for a similar computation, or \cite{T}
for a more general one). It then turns out that when $\alpha^{2}\neq1$
all foliations $\mathcal{F}_{\alpha}$, which belong to $E_{1}$,
are mutually topologically conjugate. Besides $\mathcal{F}_{\pm1}$
belong to the other equivalence class.
\begin{thm}
\emph{(Martinet-Ramis, \cite{MR}) }The quotient $E/_{ana}$ of local
analytic equivalence classes is in one-to-one correspondence with
the space $\ww C\times\left(\ww C\times c\ww C\left\{ c\right\} \right)^{2}/_{\sim}$.
The invariant of Martinet-Ramis is thus a $5$-tuple $\mathcal{M}:=\left(\mu,\tau_{0},\varphi_{0},\tau_{1},\varphi_{1}\right)$
where $\mu:=\frac{\partial^{2}R}{\partial x\partial y}\left(0,0\right)$
is the formal invariant, $\tau_{j}$ are scalars and $\varphi_{j}$
are germs of a vanishing holomorphic function at $0$, \emph{modulo}
the equivalence relation (with evident notations) : $\mathcal{M}\sim\tilde{\mathcal{M}}$
if and only if $\mu=\tilde{\mu}$, $\tau_{j}=\lambda\tilde{\tau}_{j+k}$,
$\varphi_{j}\left(c\right)=\tilde{\varphi}_{j+k}\left(\lambda c\right)$
for some $\lambda\in\ww C_{\neq0}$ not depending on $j\in\ww Z/2$
and for some $k\in\ww Z/2$. \end{thm}
\begin{rem}
The space $E_{2}$ coincides with the space of foliations such that
$\tau_{0}=\tau_{1}=0$.
\end{rem}
The same computations as above yields that $\alpha$ is an analytic
invariant, as we will see in the following section. More precisely
one can choose $\mathcal{M}$ as follows :
\begin{eqnarray*}
\tau_{0} & := & 1+\alpha\\
\tau_{1} & := & 1-\alpha\\
\varphi_{j} & := & 0\,.
\end{eqnarray*}
Hence $\mathcal{F}_{\alpha}$ and $\mathcal{F}_{\beta}$ are always
topologically conjugate (under the hypothesis $\alpha,\beta\notin\left\{ -1,1\right\} $)
whereas they are analytically conjugate if, and only if, $\alpha=\beta$.
\begin{rem}
~
\begin{enumerate}
\item These invariants are not the ``genuine'' Martinet-Ramis invariants,
which are more conventionally seen as gluing maps in the sectorial
space of leaves, meaning the diffeomorphisms :
\begin{eqnarray*}
\psi_{j}^{\infty}\,:\, c & \mapsto & c+\tau_{j}\\
\psi_{j}^{0}\,:\, c & \mapsto & c\exp\left(i\pi\mu+\varphi_{j}\left(c\right)\right)\,.
\end{eqnarray*}
In the case where $\chi\in E_{2}$, \emph{i.e. }$\psi_{j}^{\infty}=Id$,
its weak holonomy is analytically conjugate to $\psi_{0}^{0}\circ\psi_{1}^{0}$,
which is a map tangent to $e^{2i\pi\mu}Id$.
\item Elizarov's topological moduli space of $E_{1}$ is the set of all
pairs $\left(\varepsilon_{0},\varepsilon_{1}\right)\in\left\{ 0,1\right\} ^{2}\backslash\left\{ \left(0,0\right)\right\} $
such that $\varepsilon_{j}=0$ if, and only if, $\tau_{j}=0$ where
$\left(1,0\right)$ and $\left(0,1\right)$ are identified.
\end{enumerate}
\end{rem}
\bigskip{}

We propose here to build explicitly a local topological conjugacy
between $\mathcal{F}_{\alpha}$ and $\mathcal{F}_{0}$ when $\left|\alpha\right|<\frac{1}{10}$.
In fact one could achieve the same construction for any value of $\alpha$
but for the sake of concision we only retain this case. Before doing
so we begin with describing the setting for any value of $\alpha$.

\subsection{The sectorial decomposition and induced homemorphisms in the spaces
of leaves}

We split $\ww C^{2}$ into three parts : 
\begin{eqnarray*}
\ww C^{2} & = & \mathcal{V}^{+}\cup\mathcal{V}^{-}\cup\left\{ x=0\right\} 
\end{eqnarray*}
where the sectors $\mathcal{V}^{\pm}$ are, as usual, defined by 
\begin{eqnarray*}
\mathcal{V}^{\pm} & := & \left\{ \left(x,y\right)\,:\,\left|\arg x\mp\frac{\pi}{2}\right|<\frac{3\pi}{4}\right\} \,.
\end{eqnarray*}
 We denote $\mathcal{F}_{\alpha}$ the foliation induced on $\cp$
by $\omega_{\mathcal{\alpha}}$ and define $\mathcal{F}_{\alpha}^{\pm}$
as the restriction of $\mathcal{F}_{\alpha}$ to $\mathcal{V}^{\pm}$\emph{.}
We let $y_{\alpha,c}^{\pm}$ be the general solution of the differential
equation $\omega_{\alpha}=0$ for $c\in\ww C$ :
\begin{eqnarray}
y_{\alpha,c}^{\pm}\,:\, x\in\Pi\left(\mathcal{V}^{\pm}\right) & \mapsto & \exp\left(-\frac{1}{2x^{2}}\right)\left(c-\int_{\pm0i}^{x}\left(\frac{1}{i\pi}+\frac{\alpha}{i\sqrt{2\pi}}z\right)\exp\left(\frac{1}{2z^{2}}\right)\frac{\pp z}{z}\right)\label{eq:sol_gene}
\end{eqnarray}
which are holomorphic functions. The integration here is done over
a path linking $x$ to $0$ in $\Pi\left(\mathcal{V}^{\pm}\right)$
and tangent to the half-line $\pm i\ww R_{\geq0}$ at $0$. Notice
that the intersection $\mathcal{V}^{+}\cap\mathcal{V}^{-}$ is included
in the node-part $\left\{ Re\left(x^{-2}\right)>\varepsilon>0\right\} $of
the saddle-node singularity of $\mathcal{F}_{\alpha}$, meaning that
any leaf of $\mathcal{F}_{\alpha}^{\pm}$, or of its restriction to
any polydisc $\Delta$ centered at $\left(0,0\right)$, accumulates
on $\left(0,0\right)$ over these sectors. On the contrary only one
leaf accumulates on $\left(0,0\right)$ in the saddle-part $\left\{ Re\left(x^{-2}\right)<-\varepsilon<0\right\} $.
It is the leaf corresponding to $y_{\alpha,0}^{\pm}$ and we will
call it \emph{the sectorial weak separatrix}. Martinet-Ramis invariants
measure how going from one sector $\mathcal{V}^{\pm}$ to the other
changes the value of $c$ while remaining on the same global leaf
of $\mathcal{F}_{\alpha}$. In the special case we are considering
they simply consist in the Stokes coefficients of the linear differential
equation $\omega_{\alpha}=0$. One can easily check that 
\begin{eqnarray*}
\left(\forall Re\left(x\right)<0\right)\, y_{\alpha,c}^{+}\left(x\right) & = & y_{\alpha,c+1+\alpha}^{-}\left(x\right)\\
\left(\forall Re\left(x\right)>0\right)\, y_{\alpha,c}^{-}\left(x\right) & = & y_{\alpha,c+1-\alpha}^{+}\left(x\right)
\end{eqnarray*}
 so that 
\begin{eqnarray*}
\tau_{0} & = & 1+\alpha\\
\tau_{1} & = & 1-\alpha\,.
\end{eqnarray*}
Indeed the value of the difference $y_{\alpha,c}^{+}-y_{\alpha,c}^{-}$
can be obtained through Hankel's integral representation of $\frac{1}{\Gamma}$
:
\begin{eqnarray*}
\int_{\gamma_{j}}z^{a}\exp\left(\frac{1}{2z^{2}}\right)\frac{\pp z}{z^{3}} & = & -\frac{2i\pi}{\Gamma\left(a/2\right)}\left(\frac{1}{2}\right)^{a/2}\left(-1\right)^{aj}
\end{eqnarray*}
where $\gamma_{j}$ is a circle tangent at $0$ to $i\ww R$ centered
at $\left(-1\right)^{j}$. 

We can assume without loss of generality that, up to changing slightly
the aperture of the source and target sectors, $\varphi\left(\mathcal{V}^{\pm}\cap\Delta\right)\subset\mathcal{V}^{\pm}$.
Hence, following the same argument as before, any homeomorphism $\varphi$
conjugating $\mathcal{F}_{\alpha}$ and $\mathcal{F}_{0}$ on some
polydisc $\Delta$ induces (unique) homeomorphisms $\psi^{\pm}$ from
the sectorial spaces of leaves of $\mathcal{F}_{\alpha}^{\pm}$ to
the sectorial space of leaves of $\mathcal{F}_{0}^{\pm}$. Since $\varphi$
must send sectorial weak separatrices of $\mathcal{F}_{\alpha}$ onto
those of $\mathcal{F}_{0}$ the homeomorphisms $\psi^{\pm}$ shall
fix $0$ and we derive 
\begin{eqnarray*}
\psi^{\pm}\,:\,\ww C & \to & \ww C\\
c & \mapsto & \psi^{\pm}\left(c\right)\\
0 & \mapsto & 0
\end{eqnarray*}
 such that 
\begin{eqnarray*}
\varphi\left(\left\{ y=y_{\alpha,c}^{\pm}\left(x\right)\right\} \right) & \subset & \left\{ y=y_{0,\psi^{\pm}\left(c\right)}^{\pm}\left(x\right)\right\} 
\end{eqnarray*}
 and $\psi^{\pm}$ conjugate the actions of $c\mapsto c+\tau_{j}$.
See figure \ref{fig:leaves_homeo}. They will be called \emph{transverse
homeomorphisms} in the sequel as they completely determine the change
in the transverse structure of the foliations.

On the converse our aim in the rest of Section~\ref{sec:Local} is
to build special transverse homeomorphims conjugating the Stokes translations
in order that they be realized in the $\left(x,y\right)$-space by
a local homeomorphism $\varphi$. We will later extend it to the whole
$\cp$ in Section~\ref{sec:Extending}.

\begin{figure}
\includegraphics[height=5cm]{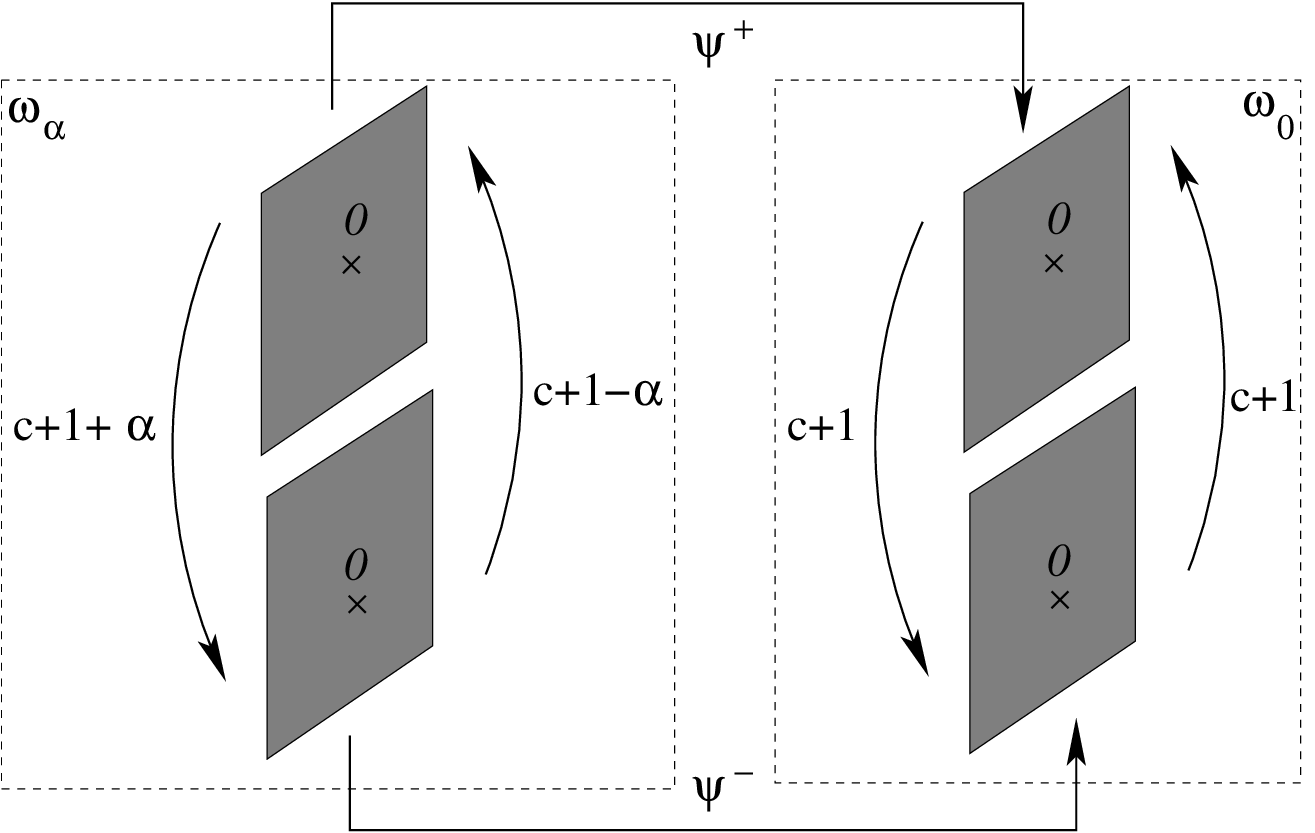}

\caption{\label{fig:leaves_homeo}The induced homeomorphisms $\psi^{+}$ and
$\psi^{-}$ between spaces of leaves}

\end{figure}

\subsection{The construction on $\ww C^{2}$}

The strategy to build such a $\varphi$ consists in the following
three steps.
\begin{enumerate}
\item Finding two transverse homeomorphisms $\psi^{\pm}$ such that
\begin{eqnarray}
\begin{cases}
\psi^{+}\left(c+1-\alpha\right) & =\psi^{-}\left(c\right)+1\\
\psi^{-}\left(c+1+\alpha\right) & =\psi^{+}\left(c\right)+1\\
\psi^{+}\left(0\right)=\psi^{-}\left(0\right) & =0\\
\lim_{c\to0,\infty}\frac{\psi^{\pm}\left(c\right)}{c} & =1
\end{cases} &  & \,.\label{eq:conj_leaves}
\end{eqnarray}
Notice the three first conditions are necessarily satisfied by any
pair of induced homeomorphisms. Besides the first two relations imply
that $\psi^{\pm}=Id+\eta^{\pm}$ where $\eta^{\pm}$ is $2$-periodic. 
\item Finding a lift $\varphi^{\pm}$ of $\psi^{\pm}$ in the ambient space
$\mathcal{V}^{\pm}\cap\overline{\ww D}\times\ww C$ such that $\varphi^{+}=\varphi^{-}$
in $\mathcal{V}^{+}\cap\mathcal{V}^{-}$ (thus defining a homeomorphism
$\varphi$ on $\overline{\ww D}_{\neq0}\times\ww C$).
\item Ensuring that $\varphi$ extends continuously to $\left\{ x=0\right\} $.
For this we need the last condition in the above system.
\end{enumerate}
To underline that fulfilling these three conditions is tricky we first
prove the
\begin{prop}
\label{pro:tricky}Assume that $\varphi$ is a local topological conjugacy
between $\mathcal{F}_{\alpha}$ and $\mathcal{F}_{0}$ such that $\varphi$
preserves globally the fibers of $\Pi$. Then $\alpha=0$.
\end{prop}
The meaning of this statement is that, unlike the analytical setting
(see \cite{MR}), topological conjugacies between non-analytically
conjugate saddle-node foliations cannot be chosen of the form $\left(x,y\right)\mapsto\left(x,Y\left(x,y\right)\right)$,
nor even of the form $\left(x,y\right)\mapsto\left(X\left(x\right),Y\left(x,y\right)\right)$,
which seriously complicates matters as we will see.
\begin{proof}
Let $\Delta$ be a polydisc on which $\varphi$ is realized. For any
$\omega\in\ww C_{\neq0}$ there exists a sequence $\left(x_{n}\right)_{n}\subset\mathcal{V}^{+}$
such that $\exp\left(-\frac{1}{2x_{n}^{2}}\right)=\omega$ and $\left(x_{n}\right)_{n}$
converges towards $0$ ; let $y$ be given in order that $\left(x_{n},y\right)_{n}\subset\Delta$.
By assumption $\varphi$ takes the form $\varphi\left(x_{n},y\right)=\left(X\left(x_{n}\right),Y\left(x_{n},y\right)\right)$.
Since
\begin{eqnarray}
y_{\alpha,c}^{+}\left(x\right)-y_{\alpha,0}^{+}\left(x\right) & = & c\exp\left(-\frac{1}{2x^{2}}\right)\label{eq:eq_lin_alpha}
\end{eqnarray}
does not depend on $\alpha$ we deduce, by setting $c:=y\omega^{-1}$,
\begin{eqnarray*}
Y\left(x_{n},y\right) & = & y_{0,0}^{+}\left(X\left(x_{n}\right)\right)+\frac{\psi^{+}\left(y\omega^{-1}\right)}{y}\left(y-y_{\alpha,0}^{+}\left(x_{n}\right)\right)\exp\left(-\frac{1}{2X\left(x_{n}\right)^{2}}\right)\,.
\end{eqnarray*}
In particular the sequence $\left(\exp\left(-\frac{1}{2}X\left(x_{n}\right)^{-2}\right)\right)_{n}$
converges towards some complex number $\lambda^{+}\left(\omega\right)\in\ww C$.
Therefore 
\begin{eqnarray}
\varphi\left(0,y\right) & = & \left(0,\lambda^{+}\left(\omega\right)\psi^{+}\left(y\omega^{-1}\right)\right)\label{eq:prolong_lin_to_sep}
\end{eqnarray}
and $\lambda^{+}\left(\omega\right)\neq0$. Obviously the same construction
can be carried out on $\mathcal{V}^{-}$, giving rise to a non-zero
vanishing function $\omega\mapsto\lambda^{-}\left(\omega\right)$
such that for all $y$ one has the same relation $\varphi\left(0,y\right)=\left(0,\lambda^{-}\left(\omega\right)\psi^{-}\left(y\omega^{-1}\right)\right)$.
Hence $\lambda^{+}\left(\omega\right)\psi^{+}\left(y\omega^{-1}\right)=\lambda^{-}\left(\omega\right)\psi^{-}\left(y\omega^{-1}\right)$
for every $\left(0,y\right)\in\Delta$. Now we fix $\omega$ small
enough in order that $\left|\omega\right|^{-1}\Delta$ contain the
points $\left(0,y\right)\in\Delta$ and $\left(0,y+2\right)$ for
at least one value of $y$. Using the $2$-periodicity of $\psi^{\pm}-Id$
we derive 
\begin{eqnarray*}
\lambda^{+}\left(\omega\right)\psi^{+}\left(y+2\right) & = & \lambda^{-}\left(\omega\right)\psi^{-}\left(y+2\right)+2\left(\lambda^{+}\left(\omega\right)-\lambda^{-}\left(\omega\right)\right)\,,
\end{eqnarray*}
meaning $\lambda^{+}\left(\omega\right)=\lambda^{-}\left(\omega\right)$.
Therefore the first two conditions of \eqref{eq:conj_leaves} yields
$\psi^{+}\left(c+1+\alpha\right)=\psi^{+}\left(c+1-\alpha\right)$
for all $c\in\ww C$. Since $\psi^{+}$ is one-to-one the only possibility
is $\alpha=0$.
\end{proof}

\subsection{The transverse homeomorphisms}

To go back to our purpose we first find admissible $\psi^{\pm}$.

\begin{figure}
\includegraphics[width=6cm]{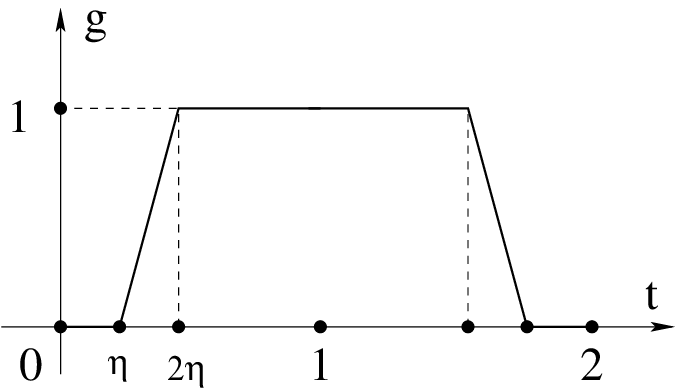}

\caption{The function $g$.}
\end{figure}

\begin{prop}
Let $\left|\alpha\right|<\frac{1}{10}$ and $\eta:=\left(1-\left|Re\left(\alpha\right)\right|\right)/3$.
We define $g$ as to be the simplest piece-wise real affine map $g$
on $\left[0,2\right]$ such that $g|_{\left[0,\eta\right]}:=0$, $g|_{\left[2\eta,2-2\eta\right]}:=1$
and $g|_{\left[2-\eta,2\right]}:=0$. We still denote by $g$ its
$2$-periodic extension to $\ww R$. Then the following functions
\begin{eqnarray*}
c\mapsto\psi^{+}\left(c\right) & := & c+\alpha g\left(Re\left(c\right)\right)\\
c\mapsto\psi^{-}\left(c\right) & := & c-\alpha+\alpha g\left(1+Re\left(c-\alpha\right)\right)
\end{eqnarray*}
form a pair of homeomorphisms solution to \eqref{eq:conj_leaves}.
Moreover for all $c\in\ww C$ : 
\begin{eqnarray*}
\left|\frac{\psi^{\pm}\left(c\right)}{c}-1\right|< & \min\left(\frac{1}{6},\frac{1}{10\left|Re\left(c\right)\right|}\right) & .
\end{eqnarray*}
\end{prop}
\begin{proof}
The fact that $\left(\psi^{+},\psi^{-}\right)$ is solution to the
system \eqref{eq:conj_leaves} is clear enough. Besides $\left|\alpha\right|<\frac{1}{10}$
so 
\begin{eqnarray*}
\sup_{t\in\left[-1,1\right]}\left|\frac{g\left(t\right)}{t}\right| & \leq & \frac{1}{2\eta}<\frac{5}{3}\,,
\end{eqnarray*}
while for $\left|t\right|\geq1$ one has $g\left(t\right)\leq1$.
Hence
\begin{eqnarray*}
\left|\frac{\psi^{+}\left(c\right)}{c}-1\right| & < & \left|\alpha\right|\min\left(\frac{5}{3},\frac{1}{\left|Re\left(c\right)\right|}\right)\,.
\end{eqnarray*}
The same kind of estimate arises for the other map, where if $\left|Re\left(c\right)\right|\in\left[0,1\right]$
:
\begin{eqnarray*}
\left|\frac{\psi^{-}\left(c\right)}{c}-1\right| & = & \left|\frac{\alpha}{c}\left(g\left(1-Re\left(\alpha\right)+Re\left(c\right)\right)-1\right)\right|\\
 & < & \frac{5}{3}\left|\alpha\right|<\frac{1}{6}
\end{eqnarray*}
since $1+Re\left(\alpha\right)-\eta>\frac{3}{5}$. To end the proof
we only need to notice that $\psi^{\pm}-Id$ is $\frac{1}{6}$-Lipschitz,
thus one-to-one and onto $\ww C$.
\end{proof}

\subsection{The homeomorphism $\varphi$}

We look now for $\varphi$ on $\ww{\overline{D}}\times\ww C$ as we'll
extend it to the entire projective plane in the next section. As was
noticed in Proposition~\ref{pro:tricky} we cannot preserve globally
the $x$-variable in the four directions $\left\{ \cos\arg x^{2}=0\right\} $.
Hence we build a new (sectorial) variable $X^{\pm}\left(x,y\right)$
which will mostly be the identity except in the neighbourhood of those
forbidden directions. We define
\begin{eqnarray*}
X^{\pm}\left(x,y\right) & := & x\left(1-2x^{2}\log f^{\pm}\left(x,c^{\pm}\left(x,y\right)\right)\right)^{-1/2}
\end{eqnarray*}
where $f^{\pm}$ is a functional parameter which will be adjusted
in the sequel to suit our needs, and where $c^{\pm}$ is the function
holomorphic on $\mathcal{V}^{\pm}$ defined by the relation
\begin{eqnarray*}
y_{\alpha,c^{\pm}\left(x,y\right)}\left(x\right) & = & y\,.
\end{eqnarray*}
For any fixed $x$ the partial function $y\mapsto c^{\pm}\left(x,y\right)$
is a diffeomorphism of $\ww C$.

In order to send a leave of $\mathcal{F}_{\alpha}^{\pm}$ into a leave
of $\mathcal{F}_{0}^{\pm}$ while changing the transverse structure
we take the new $y$-variable as being the following
\begin{eqnarray*}
Y^{\pm}\left(x,y\right) & := & y_{0,0}^{\pm}\left(X^{\pm}\left(x,y\right)\right)+f^{\pm}\left(x,c^{\pm}\left(x,y\right)\right)\psi^{\pm}\left(c^{\pm}\left(x,y\right)\right)\exp\left(-\frac{1}{2x^{2}}\right)
\end{eqnarray*}
and set 
\begin{eqnarray}
\varphi^{\pm}\left(x,y\right) & := & \left(X^{\pm}\left(x,y\right),Y^{\pm}\left(x,y\right)\right)\,.\label{eq:phi}
\end{eqnarray}
If we want that $\varphi^{+}$ and $\varphi^{-}$ glue on each connected
component of $\mathcal{V}^{+}\cap\mathcal{V}^{-}$ we must require
that for all $c\in\ww C$ : 
\begin{eqnarray*}
\begin{cases}
f^{+}\left(x,c\right)=f^{-}\left(x,c+1+\alpha\right) & \,\,\,,\,\forall Re\left(x\right)<0\\
f^{-}\left(x,c\right)=f^{+}\left(x,c+1-\alpha\right) & \,\,\,,\,\forall Re\left(x\right)>0
\end{cases} &  & .
\end{eqnarray*}
If we moreover wish that $\varphi^{\pm}$ extend to $Id$ on $\left\{ x=0\right\} $
the parameters $f^{\pm}$ must satisfy the additional condition that
for all $y_{0}$ : 
\begin{eqnarray*}
\lim_{\left(x,y\right)\to\left(0,y_{0}\right)}f^{\pm}\left(x,c^{\pm}\left(x,y\right)\right)\frac{\psi^{\pm}\left(c^{\pm}\left(x,y\right)\right)}{c^{\pm}\left(x,y\right)} & = & 1\,.
\end{eqnarray*}
 The following lemma is straightforward to prove :
\begin{lem}
\label{lem:local_homeo}Let $\left(\chi_{1},\chi_{2}\right)$ be the
simplest non-negative affine partition of unity of the circle $\ww S^{1}\simeq\ww R_{/2\pi\ww Z}$
such that $\chi_{1}\left(\frac{\pi}{4}+k\frac{\pi}{2}\right)=1$ for
any $k\in\ww Z_{/4\ww Z}$, $\chi_{2}\left(\theta\right)=1$ whenever
$\left|\cos\left(2\theta\right)\right|>\delta$ for some small, fixed
$\delta>0$ and $\chi_{1}+\chi_{2}=1$. Define the functions
\begin{eqnarray*}
f^{\pm}\left(x,c\right) & := & \chi_{1}\left(\arg x\right)\frac{c}{\psi^{\pm}\left(c\right)}+\chi_{2}\left(\arg x\right)\,.
\end{eqnarray*}

These functions satisfy the following properties :
\begin{enumerate}
\item $\left|f^{\pm}\left(x,c\right)-1\right|<\frac{1}{5}$ and $\left|2x^{2}\log f^{\pm}\left(x,c\right)\right|<\frac{3}{5}\left|x\right|^{2}$
whenever $\left|x\right|\leq1$,
\item $f^{\pm}$ is continuous on $\ww{\overline{D}}\times\ww C$,
\item $f^{\pm}$ is constant to $1$ on $\mathcal{V}^{+}\cap\mathcal{V}^{-}$, 
\item $\lim_{\left(x,y\right)\to\left(0,y_{0}\right)}f^{\pm}\left(x,c^{\pm}\left(x,y\right)\right)\frac{\psi^{\pm}\left(c^{\pm}\left(x,y\right)\right)}{c^{\pm}\left(x,y\right)}=1$
for all $y_{0}\in\ww C$,
\item for all fixed $x$ the maps $c\mapsto f^{\pm}\left(x,c\right)\psi^{\pm}\left(c\right)$
are homeomorphisms of the complex line.
\end{enumerate}
\end{lem}
As a consequence the map $\varphi$ thus defined is a continuous map
from $\ww{\overline{D}}\times\ww C$ conjugating the foliations $\mathcal{F}_{\alpha}$
and $\mathcal{F}_{0}$ on this domain.
\begin{prop}
The map $\varphi$ is one-to-one and thus defines a homeomorphism
from $\overline{\ww D}\times\ww C$ onto its image $W\times\ww C$
which, up to rescaling $\varphi$ in the first coordinate for both
source and target spaces, contains $\overline{\ww D}\times\ww C$.\end{prop}
\begin{proof}
Firstly we shall prove the latter claim. Let us write $\varphi=\left(X,Y\right)$.
Because of the first statement of the previous lemma we have 
\begin{eqnarray*}
\left|\frac{X\left(x,y\right)}{x}-1\right| & \leq & A\left|x\right|^{2}
\end{eqnarray*}
for some $A>0$ and all $\left|x\right|\leq1$. This implies that
for $\left|x\right|<\delta$ small enough $\varphi\left(\mathcal{V}^{\pm}\cap\delta\ww D\times\ww C\right)$
contains a sector $W^{\pm}:=\left\{ x\,:\,\left|x\right|<r\,,\,\left|\arg x\mp\frac{\pi}{2}\right|<\frac{3\pi}{4}-\theta\right\} $,
where $\theta$ can be chosen as small as we wish by decreasing $\delta$.
Up to rescalling the $x-$ and $X-$coordinates we can then assume
that $\overline{\ww D}\subset X\left(\overline{\ww D}\times\ww C\right)$.
Because of (5) we also derive that $\varphi\left(\mathcal{V}^{\pm}\cap\delta\ww D\times\ww C\right)$
contains a sector of lesser aperture $W^{\pm}\times\ww C$, so that
$\overline{\ww D}\times\ww C\subset\varphi\left(\overline{\ww D}\times\ww C\right)$
as required.

To prove that $\varphi$ is one-to-one we first notice that since
$\varphi$ preserves the sector decomposition of $\overline{\ww D}\times\ww C$
and since each leaf of the sectorial foliations $\mathcal{F}_{\alpha}^{\pm}$
is the graph of a function holomorphic on $W^{\pm}$, we only need
to prove that the restriction of $\varphi^{\pm}$ to some transversal
$\left\{ x=x_{0}\right\} $ is one-to-one. Let us choose $x_{0}:=\pm1$,
so that $f^{\pm}\left(x_{0},c\right)=1$. We thus have that 
\begin{eqnarray*}
\varphi^{\pm}\left(x_{0},y\right) & = & \left(x_{0},y_{0,0}\left(x_{0}\right)+\psi^{\pm}\left(c^{\pm}\left(x_{0},y\right)\right)\exp\left(-\frac{1}{2x_{0}^{2}}\right)\right)\,,
\end{eqnarray*}
which completes the proof.
\end{proof}

\section{\label{sec:Extending}Extending the homeomorphism}

The foliations under consideration have exactly three singularities,
located in homogeneous coordinates at $\left[0:0:1\right]$, $\left[0:1:0\right]$
and $\left[1:0:0\right]$. We will use the following three affine
charts of $\cp$ :
\begin{eqnarray*}
\ww C^{2} & = & \left\{ \left(x,y\right)\right\} =\left\{ \left[x:y:1\right]\right\} \\
\ww C^{2} & = & \left\{ \left(s,t\right)\right\} =\left\{ \left[1:t:s\right]\right\} \\
\ww C^{2} & = & \left\{ \left(u,v\right)\right\} =\left\{ \left[u:1:v\right]\right\} 
\end{eqnarray*}
with transition maps
\begin{eqnarray*}
y=tx\,,\,1=sx\,,\, x=uy\,,\,1=vy\,,\, v=us\,,\,1=ut & \,.
\end{eqnarray*}

The singular point $\left[0:0:1\right]$ has been extensively studied
in the previous sections and this study gave rise to a topological
conjugacy between the foliations $\mathcal{F}_{0}$ and $\mathcal{F}_{\alpha}$
from $\overline{\ww D}\times\ww C$ onto its image.
\begin{lem}
One can extend $\varphi$ to a homeomorphism of $\ww C\times\ww C$
such that the extension, still noted $\varphi$, preserves each fiber
of $\Pi$ outside $\overline{\ww D}\times\ww C$ and still conjugates
$\mathcal{F}_{0}$ and $\mathcal{F}_{\alpha}$.\end{lem}
\begin{proof}
This result is straightforward. Choose $0<r<1$ and consider the simplest
affine non-negative partition of unity $\left(\xi_{1},\xi_{2}\right)$of
$\ww R_{\geq0}$ where $\xi_{1}=1$ on $\left[0,r\right]$, $\xi_{2}=1$
outside $\left[0,1\right]$ and $\xi_{1}+\xi_{2}=1$. By setting 
\begin{eqnarray*}
\left(\forall\left(x,c\right)\in\overline{\ww D}\times\ww C\right)\,\,\,\,\hat{f}^{\pm}\left(x,c\right) & := & f^{\pm}\left(x,c\right)\xi_{1}\left(\left|x\right|\right)+\xi_{2}\left(\left|x\right|\right)\\
\left(\forall\left(x,c\right)\notin\overline{\ww D}\times\ww C\right)\,\,\,\,\hat{f}^{\pm}\left(x,c\right) & := & 1
\end{eqnarray*}
 and defining $\varphi$ by \eqref{eq:phi} the reader can easily
show that our claim is true, as in Proposition~\ref{lem:local_homeo}.
\end{proof}
Once this is stated we have to check that $\varphi$ extends to a
homeomorphism of $\cp$. It is only a matter of writing things in
appropriate charts since as expected $\varphi$ extends to $Id$ along
the line at infinity $\left\{ s=0\right\} \cup\left\{ v=0\right\} $.
\begin{prop}
\label{pro:global_homeo}$\varphi$ extends to a global homeomorphism
of $\cp$, which implies that the main theorem is true.\end{prop}
\begin{proof}
Let us write $\varphi$ in the chart $\left(s,t\right)$ near $\left\{ 0\right\} \times\ww C$
:
\begin{eqnarray*}
\tilde{\varphi}\left(s,t\right) & := & \left(s,s\left(y_{0,0}^{\pm}\left(\frac{1}{s}\right)+\psi^{\pm}\left(\tilde{c}^{\pm}\left(s,t\right)\right)\exp\left(-\frac{s^{2}}{2}\right)\right)\right)
\end{eqnarray*}
where as before $\tilde{c}^{\pm}$ is uniquely defined on $\mathcal{V}^{\pm}$
by 
\begin{eqnarray*}
y_{\alpha,\tilde{c}^{\pm}\left(s,t\right)}^{\pm}\left(\frac{1}{s}\right) & = & \frac{t}{s}\,,
\end{eqnarray*}
 for all $\left(s,t\right)\in\ww C_{\neq0}\times\ww C$. Following
\eqref{eq:sol_gene} we have 
\begin{eqnarray*}
y_{\alpha,0}^{\pm}\left(\frac{1}{s}\right)=-\exp\left(-\frac{s^{2}}{2}\right)\int_{s}^{\pm\infty i}\left(1+\frac{\alpha}{z}\right)\exp\left(\frac{z^{2}}{2}\right)\frac{\mbox{d}z}{z}
\end{eqnarray*}
so that 
\begin{eqnarray*}
\lim_{s\to0}sy_{\alpha,0}^{\pm}\left(\frac{1}{s}\right) & = & 0\,.
\end{eqnarray*}
On the other hand if $\left(s,t\right)\in\mathcal{V}^{\pm}$ then
$t=s\left(y_{0,0}^{\pm}\left(\frac{1}{s}\right)+\tilde{c}^{\pm}\left(s,t\right)\exp\left(-s^{2}/2\right)\right)$
so that setting
\begin{eqnarray*}
\tilde{\varphi}\left(0,t\right) & := & \left(0,t\right)
\end{eqnarray*}
 defines a continuous extension of $\tilde{\varphi}$ to $\ww C\times\ww C$,
because $\psi^{\pm}\left(\tilde{c}^{\pm}\left(s,t\right)\right)-\tilde{c}^{\pm}\left(s,t\right)$
remains bounded as $\left(s,t\right)\to\left(0,t_{0}\right)$. 

Finally we shall check that $\varphi$ admits a limit at $\left[0:1:0\right]$.
We write it in the chart $\left(u,v\right)$ as 
\begin{eqnarray*}
\hat{\varphi}\left(u,v\right) & :=\left(\hat{U}^{\pm}\left(u,v\right),\hat{V}^{\pm}\left(u,v\right)\right)
\end{eqnarray*}
where 
\begin{eqnarray}
\hat{X}^{\pm}\left(u,v\right) & = & \frac{u}{v\sqrt{1-2\frac{u^{2}}{v^{2}}\log f^{\pm}\left(\frac{u}{v},\hat{c}^{\pm}\left(u,v\right)\right)}}\nonumber \\
\hat{V}^{\pm}\left(u,v\right) & = & \frac{1}{y_{0,0}^{\pm}\left(\hat{X}^{\pm}\left(u,v\right)\right)+f^{\pm}\left(\frac{u}{v},\hat{c}^{\pm}\left(u,v\right)\right)\psi^{\pm}\left(\hat{c}^{\pm}\left(u,v\right)\right)\exp\left(-\frac{v^{2}}{2u^{2}}\right)}\label{eq:V}\\
\hat{U}^{\pm}\left(u,v\right) & = & \hat{V}^{\pm}\left(u,v\right)\hat{X}^{\pm}\left(u,v\right)\nonumber 
\end{eqnarray}
and $\hat{c}^{\pm}\left(u,v\right)$ is uniquely defined by $v^{-1}=y_{\alpha,0}\left(u/v\right)+\hat{c}^{\pm}\left(u,v\right)\exp\left(-v^{2}/2u^{2}\right)$
for $\left(\frac{u}{v},\frac{1}{v}\right)\in\mathcal{V}^{\pm}$. Let
us split $\ww C_{\neq0}^{2}$ into the sets 
\begin{eqnarray*}
\mathcal{C}^{>}: & = & \left\{ \left(u,v\right)\,:\,\left|\frac{u}{v}\right|>1\right\} \\
\mathcal{C}^{\leq} & := & \left\{ \left(u,v\right)\,:\,\left|\frac{u}{v}\right|\leq1\right\} \,.
\end{eqnarray*}
We recall that $f^{\pm}\left(\frac{u}{v},\hat{c}^{\pm}\left(u,v\right)\right)=1$
whenever $\left(u,v\right)\in\mathcal{C}^{>}$. On the other hand
$\hat{X}^{\pm}\left(\mathcal{C}^{\leq}\right)$ is bounded. Hence
it suffices to show that $\hat{V}^{\pm}\left(u,v\right)=O\left(v\right)$
as $\left(u,v\right)\to0$ in order to prove that $\hat{\varphi}$
extends continuously to $\left(0,0\right)$ by $\hat{\varphi}\left(0,0\right):=\left(0,0\right)$. 

Because $x\mapsto y_{0,0}^{\pm}\left(x\right)$ is smooth as a real
map there exists a constant $A>0$ such that for all $\left(u,v\right)\in\mathcal{C}^{<}$
one has 
\begin{eqnarray*}
\left|y_{0,0}^{\pm}\left(\hat{X}^{\pm}\left(u,v\right)\right)-y_{0,0}^{\pm}\left(\frac{u}{v}\right)\right| & \leq & A\left|\frac{u^{3}}{v^{3}}\right|
\end{eqnarray*}
whereas this estimate is true with $A:=0$ when $\left(u,v\right)\in\mathcal{C}^{>}$.
We then derive :
\begin{eqnarray*}
\left|\frac{1}{\hat{V}^{\pm}\left(u,v\right)}-\frac{1}{v}\right| & \leq & A\left|\frac{u^{3}}{v^{3}}\right|+\left|\hat{c}^{\pm}\left(u,v\right)-f^{\pm}\left(\frac{u}{v},\hat{c}^{\pm}\left(u,v\right)\right)\psi^{\pm}\left(\hat{c}^{\pm}\left(u,v\right)\right)\right|\left|\exp\left(-v^{2}/2u^{2}\right)\right|\\
\left|\frac{v}{\hat{V}^{\pm}\left(u,v\right)}-1\right| & \leq & A\left|u\right|\left|\frac{u}{v}\right|^{2}+B\left|v\right|\left|\exp\left(-v^{2}/2u^{2}\right)\right|
\end{eqnarray*}
for some $B>0$. Hence there exists $B^{>}>0$ such that for all $\left(u,v\right)\in\mathcal{C}^{>}$
: 
\begin{eqnarray*}
\left|\frac{v}{\hat{V}^{\pm}\left(u,v\right)}-1\right| & < & B^{>}\left|v\right|\,.
\end{eqnarray*}
Consider now $\left(u,v\right)\in\mathcal{C}^{\leq}$. Clearly there
exists $B_{+}^{\leq}>0$ such that if $Re\left(\frac{u^{2}}{v^{2}}\right)\geq0$
then 
\begin{eqnarray*}
\left|\frac{v}{\hat{V}^{\pm}\left(u,v\right)}-1\right| & < & B_{+}^{\leq}\left(\left|u\right|+\left|v\right|\right)
\end{eqnarray*}
while according to \eqref{eq:V} there exists $B_{-}^{\leq}>0$ such
that if $Re\left(\frac{u^{2}}{v^{2}}\right)<0$ then
\begin{eqnarray*}
\left|\hat{V}^{\pm}\left(u,v\right)\right| & < & B_{-}^{\leq}\left|v\right|\,.
\end{eqnarray*}
\end{proof}


\begin{thebibliography}{NSS}
\bibitem[E]{E}P. M. Elizarov, \emph{Orbital topological classification
of analytic differential equations in a neighbourhood of a degenerate
elementary singular point on the two-dimensional complex plane}, Trudy
Sem. Petrovskogo, vol.13, p. 137-165 (1988)

\bibitem[GO]{GO} X. Gómez-Mont, L. Ortíz-Bobadilla, \emph{Sistemas
dinámicos holomorfos en superficies}, Sociedad Matemática Mexicana,
México City (1989)

\bibitem[I]{I}Yu. S. Ilyashenko, \emph{Topology of phase portraits
of analytic differential equations on a complex projective plane},
Trudy Sem. Petrovskogo, vol.4, p. 83--136 (1978)

\bibitem[L]{L}N. N. Ladis, \emph{Integral curves of complex homogenuous
equations}, Differential Equations, vol.15 \#2, p. 167-171 (1979)

\bibitem[N]{N}I. Nakai, \emph{Separatrices for nonsolvable dynamics
on }$\left(\ww C,0\right)$, Annales de l'Institut Fourier, vol.44
\#2, p. 569-599 (1994)

\bibitem[NSS]{NSS}A. L. Neto, P. Sad, B. Sc\'{a}rdua, \emph{On topological
rigidity of projective foliations}, Bulletin de la Société Mathématique
de France, vol.126 \#3, p. 381-406 (1998)

\bibitem[M]{M}J.-F. Mattéi, \emph{Modules de feuilletages holomorphes
singuliers. I. Équisingularité.,} Inventiones Mathematicae, vol.103
\#2, p. 297--325 (1991)

\bibitem[MR]{MR}J. Martinet, J.-P. Ramis, \emph{Problèmes de modules
pour des équations différentielles non linéaires du premier ordre}.
Institut des Hautes Études Scientifiques, Publications Mathématiques,
vol. 55, p. 63-164 (1982)

\bibitem[S]{S}A. A. Scherbakov, \emph{Topological and analytical
conjugation of noncommutative groups of germs of conformal mappings},
Journal Soviet Mathematics, vol.35 \#6, p. 2827-2850 (1986)

\bibitem[T]{T}L. Teyssier, \emph{Examples of non-conjugated holomorphic
vector fields and foliations}, Journal of Differential Equations,
vol.205 \#2, p. 390-407 (2004)

\bibitem[V]{V}S. M. Voronin, \emph{Smooth and analytic equivalence
of germs of resonance vector fields and maps}, Usp. Mat. Nauk, vol.38
\#5, p. 111-126 (1983)\end{thebibliography}
\end{document}